\documentclass[12pt,reqno]{amsart}
\usepackage[top=3cm,bottom=3cm,left=3cm,right=3cm,a4paper]{geometry}
\usepackage[T1]{fontenc}
\usepackage{tgtermes}
\usepackage{newtxmath}
\usepackage[cal=cm]{mathalfa}
\usepackage{helvet}

\usepackage[colorlinks, linkcolor=blue, citecolor=blue, urlcolor=blue]{hyperref}
\usepackage{enumitem}
\usepackage{tikz}
\usepackage[labelfont={}]{subcaption}

\makeatletter
\@namedef{subjclassname@2010}{%
  \textup{2010} Mathematics Subject Classification}
\makeatother

\newtheorem{theorem}{Theorem}[section]
\newtheorem{corollary}[theorem]{Corollary}
\newtheorem{fact}[theorem]{Fact}
\newtheorem{lemma}[theorem]{Lemma}
\newtheorem{proposition}[theorem]{Proposition}

\newtheorem{claim}[theorem]{Claim}

\theoremstyle{definition}
\newtheorem{definition}[theorem]{Definition}
\newtheorem{remark}[theorem]{Remark}

\begin{document}
\baselineskip=17pt

\title[On the cofinality of the least $\lambda$-strongly compact cardinal]{On the cofinality of the least $\lambda$-strongly compact cardinal}

\author[Zhixing You]{Zhixing You}
\address{
Department of Mathematics and Computer Science\\
University of Barcelona\\
Barcelona 08001, Catalonia, Spain
}
\address{Institute of Mathematics\\
Academy of Mathematics and Systems Science\\
Chinese Academy of Sciences\\
Beijing 100190 \\
People’s Republic of China
}
\address{School of Mathematical Sciences \\
University of Chinese Academy of Sciences \\
Beijing 100049 \\
People’s Republic of China
}
\email{youzhixing16@mails.ucas.ac.cn}

\author[Jiachen Yuan]{Jiachen Yuan}
\address{School of Mathematics\\
University of East Anglia\\
Norwich NR4 7TJ\\
The United Kingdom}
\email{Jiachen.Yuan@uea.ac.uk}

\date{}

\begin{abstract}
In this paper, we characterize the possible cofinalities of the least $\lambda$-strongly compact cardinal.
We show that, on the one hand, for any regular cardinal, $\delta$, that carries a $\lambda$-complete uniform ultrafilter,
it is consistent, relative to the existence of a supercompact cardinal above $\delta$,
that the least $\lambda$-strongly compact cardinal has cofinality $\delta$. On the other hand, provably the cofinality of the least $\lambda$-strongly compact cardinal always carries a $\lambda$-complete uniform ultrafilter.

\end{abstract}

\subjclass[2010]{03E35, 03E55}

\keywords{$\lambda$-strongly compact cardinal, Cofinality, Iterated ultrapower, Radin forcing}
\maketitle
\section{introduction}
In \cite{BMO2014,BM2014}, Bagaria and Magidor introduced the notion of $\lambda$-strong compactness (see Definition \ref{dl}), which generalized the well-known notion of strong compactness.

$\lambda$-strong compactness shares some similarities with strong compactness. For example,
$\lambda$-strong compactness can be characterized in terms of compactness properties of infinite languages, elementary embeddings, ultrafilters, etc. (see \cite{BMO2014,BM2014,U2020}).

It turns out that this large cardinal notion is successful. $\lambda$-strong compactness, especially for the case $\lambda=\omega_1$, provides weaker large cardinal strength to prove various results known to follow from strong compactness. Furthermore, it is actually the exact large cardinal strength of some properties in different areas (see \cite{BMO2014,BM2014}). In addition, recently Goldberg proved
 a celebrated theorem$\footnote{Goldberg verified Woodin's conjecture: in second-order set theory, every two elementary embeddings $j_0,j_1:
 V \rightarrow M$ into the same inner model $M$ agree on ordinals, i.e., $j_0 \upharpoonright \mathrm{Ord}=j_1 \upharpoonright  \mathrm{Ord}$. Actually, he proved that if they don't agree on ordinals,
 then there exists an $\omega_1$-strongly compact cardinal, which implies that they agree on ordinals.
 See also https://www.youtube.com/watch?v=DesDNK1XN3M.}$ in \cite{G2021}. In his proof, the existence of $\omega_1$-strongly compact cardinal is natural and ensures that $\mathrm{SCH}$ holds above it, which implies that the theorem holds.

The least $\lambda$-strongly compact cardinal itself is of particular interest, because it may have very odd properties$\footnote{Recently, Gitik constructed a model of $\mathrm{ZFC}$, relative to the existence of a supercompact cardinal, in which the least $\lambda$-strongly compact cardinal is not strongly compact, but also stays regular. He also constructed a model of $\mathrm{ZFC}$, relative to the existence of two supercompact cardinals, in which the least $\lambda$-strongly compact cardinal is not strong limit (see \cite{Moti Gitik(2020)}).}$. Bagaria and Magidor \cite{BMO2014} showed that this cardinal must be a limit cardinal. But surprisingly, it may not be weakly inaccessible. Namely, it can be singular (see \cite{BM2014}).

However, there are still some limitations about the cofinality of the least $\lambda$-strongly compact cardinal. By a standard argument in \cite{BMO2014}, one can see the cofinality must be greater than or equal to  the least measurable cardinal.

Following the results of Bagaria-Magidor, we were very curious about the exact limitations of the cofinality of the least $\lambda$-strongly compact cardinal. Back in 2019, it was not
clear. This is how our work got started.

In this paper, Theorem \ref{t1} extends the consistency result of Bagaria-Magidor (\cite[Theorem 6.1]{BM2014}) to $\lambda$-measurable cardinals, and Proposition \ref{pro3} shows this result is optimal. As
a corollary (Corollary \ref{c2s}), it shows that relative to the existence of two supercompact
cardinals, for any regular cardinal between them, it is consistent that the cofinality of the
least $\omega_1$-strongly compact cardinal is that cardinal.

\subsection*{The structure of the paper}
The paper covers some basic technical preliminaries about $\lambda$-strongly compact cardinals, Radin forcings and iterated ultrapowers in Section \ref{sec2}. We give the main idea of the proof of our consistency result in Section \ref{sec3}. Finally, in section
\ref{sec4} we prove the consistency result and show that it is optimal.

\section{preliminaries}\label{sec2}
We use V to denote the ground model in which we work. For any ordinals $\alpha<\beta$, $[\alpha, \beta], [\alpha, \beta),(\alpha, \beta]$, and $(\alpha, \beta)$ are as in standard interval notation. Let $\mathrm{id}_{M}$ denote the class
identity function from $M$ to $M$, and we will simply write $\mathrm{id}$ when $M$ is clear from the context.  For a
sequence $u$, let $\mathrm{lh}(u)$ denote the length of $u$. For an elementary embedding $j : V \rightarrow M$
with $M$ transitive, let $\mathrm{crit}(j)$ denote the critical point of $j$.

For every $\gamma$ with Cantor normal form $\gamma=\omega^{\gamma_1}+\dots+\omega^{\gamma_n}$,
where $\gamma_1 \geq \dots \geq \gamma_n$, let $\beta_{\gamma}:=1+\gamma_n$.
By induction, one may easily show that if $\gamma$ is a successor ordinal, then $\beta_{\gamma}=1$; and
if $\gamma$ is a limit ordinal, then $\beta_{\gamma}=\limsup_{\alpha<\gamma}(\beta_{\alpha}+1)$.

For a cardinal $\theta$, a sequence $\langle \mathcal{C}_{\alpha} \ | \ \alpha<\theta^+ \rangle$ is a \emph{$\square_{\theta,\omega}$-sequence} if and only if whenever $\alpha$ is a limit ordinal with $\theta<\alpha<\theta^+$,
\begin{enumerate}
  \item $1 \leq |\mathcal{C}_{\alpha}| \leq \omega$, and
  \item for all $C \in \mathcal{C}_{\alpha}$,
  \begin{enumerate}
    \item $C$ is a club subset of $\alpha$.
    \item $C$ has order type at most $\theta$.
    \item If $\eta$ is a limit point of $C$, then $C \cap \eta \in \mathcal{C}_{\eta}$.
  \end{enumerate}
\end{enumerate}

For every $A$ with $|A| \geq \kappa$, let $\mathcal{P}_{\kappa}(A)=\{x \subseteq A \ | \ |x|<\kappa \}$.
A set $U \subseteq \mathcal{P}_{\kappa}(A)$ is a \emph{measure} if it is a non-principal $\kappa$-complete ultrafilter on  $\mathcal{P}_{\kappa}(A)$.
A measure $U$ on $\mathcal{P}_{\kappa}(A)$ is \emph{fine} if for every $x \in \mathcal{P}_{\kappa}(A)$, $\{y \in \mathcal{P}_{\kappa}(A) \ | \ x\subseteq y \} \in U$. A measure $U$ on $\mathcal{P}_{\kappa}(A)$ is \emph{normal} if for any function $f:\mathcal{P}_{\kappa}(A)
\rightarrow A$ with $\{x \in \mathcal{P}_{\kappa}(A) \ | \ f(x) \in x \} \in U$, there is a set in $U$ on which $f$ is constant.

A cardinal $\kappa$ is \emph{$\alpha$-supercompact} if there exists an elementary embedding $j:V \rightarrow M$ with $M$ transitive such that
$\mathrm{crit}(j)=\kappa$, $j(\kappa)>\alpha$ and $M$ is closed under sequences of length $\alpha$. A cardinal $\kappa$ is \emph{supercompact} if  it is $\alpha$-supercompact for every $\alpha$. Equivalently,
$\kappa$ is supercompact if and only if for every $\alpha \geq \kappa$, there is a normal fine measure on $\mathcal{P}_{\kappa}(\alpha)$
 (see \cite[22]{Kana}).

\subsection{$\lambda$-strongly compact cardinals} \label{scc}

\begin{definition}(\cite{BMO2014,BM2014}) \label{dl}
Suppose $\delta \geq \lambda$ are uncountable cardinals.
\begin{enumerate}
  \item For every $\alpha \geq \delta$, $\delta$ is \emph{$\lambda$-strongly compact up to $\alpha$}
  if there exists a definable elementary embedding $j:V \rightarrow M$ with $M$ transitive, such that
 $\mathrm{crit}(j) \geq \lambda$ and there exists a $D \in M$ such that $j''\alpha \subseteq D$ and $M \models |D|<j(\delta)$.
  \item $\delta$ is \emph{$\lambda$-strongly compact} if $\delta$ is $\lambda$-strongly compact up to $\alpha$ for every $\alpha \geq \delta$.
\end{enumerate}
\end{definition}
By the definition above, it is easy to see that if $\delta$ is $\lambda$-strongly compact, then it is $\lambda'$-strongly compact for every uncountable cardinal $\lambda'<\lambda$, and any cardinal greater tha $\delta$ is also $\lambda$-strongly compact.

We say that $\delta$ is \emph{$\lambda$-measurable} if and only if it is $\lambda$-strongly compact up to $\delta$.

Usuba gave a characterization of $\lambda$-strongly compact cardinals in terms of $\lambda$-complete uniform ultrafilters \cite[Theorem 1.2]{U2020}, which generalized a result of Ketonen. The following proposition is a simple local version of the characterization.
\begin{proposition}\label{pu}
Suppose $\delta \geq \lambda$ are uncountable regular cardinals. Then $\delta$ is $\lambda$-measurable if and only if
$\delta$ carries a $\lambda$-complete uniform ultrafilter, i.e., there is a $\lambda$-complete ultrafilter $U$ over $\delta$
such that every $A \in U$ has cardinality $\delta$.
\end{proposition}
\begin{proof}
If $\delta$ is $\lambda$-measurable, then there exists a definable elementary embedding $j:V \rightarrow M$ with $M$ transitive, such that $\mathrm{crit}(j) \geq \lambda$ and there exists a $D \in M$ so that $j''\delta \subseteq D$
 and $M \vDash |D|<j(\delta)$. Since $\delta$ is regular, we have $j(\delta)$ is regular in $M$. Hence $\sup(j''\delta)<j(\delta)$.
Now we may define a $\lambda$-complete uniform ultrafilter $U$ over $\delta$ by
 $X \in U$ if and only if $X \subseteq \delta$ and $\sup(j''\delta) \in j(X)$.

Conversely, if $\delta$ carries a $\lambda$-complete uniform ultrafilter, say $U$, then the canonical embedding $j_U:V \rightarrow M_U \cong \mathrm{Ult}(V,U)$ satisfies $\mathrm{crit}(j_U) \geq \lambda$ and $\sup(j''\delta) \leq [\mathrm{id}]_U<j(\delta)$. Thus $j_U$ witnesses that $\delta$ is $\lambda$-measurable.
\end{proof}

\begin{theorem}\label{lem3}(\cite{BMO2014})
The least $\lambda$-strongly compact cardinal is a limit cardinal.
\end{theorem}

\subsection{Radin forcing}\label{rf} In this subsection, we will generally follow \cite[Section 6.1]{BM2014} for the
presentation of Radin forcing. For the sake of completeness, we also review its definition
and some related basic properties, including the coherence of measure sequences (Lemma
\ref{p1}), the characterization of Radin generic objects via the geometric conditions (Theorem \ref{itu}), the construction of Radin generic objects via iterated ultrapowers (Theorem \ref{thm1}). For
the readers’ convenience, we also give proofs for some of these properties. Readers who
are familiar with Radin forcing may skip these details.

We first define \emph{measure sequences}, which are the building blocks of Radin forcing.
\begin{definition}
A non-empty sequence $u=\langle u(\alpha) \ | \ \alpha<\mathrm{lh}(u) \rangle$ is a \emph{measure sequence}
 if there exists a definable elementary embedding $j:V \rightarrow M$ with $M$ transitive such that $u(0)$ is a measurable cardinal, and for each $\alpha$ with $0<\alpha<\mathrm{lh}(u)$, $u \upharpoonright \alpha \in M$ and $u(\alpha)=\{A \subseteq V_{u(0)} \ | \ u \upharpoonright \alpha \in j(A)\}$.
\end{definition}

For simplicity of notation, we write $\kappa(u)$ for $u(0)$, $\mathcal{F}(u)$ for $\bigcap_{0<\alpha<\mathrm{lh}(u)}u(\alpha)$ if the length of $u$ is greater than $1$, and $\mathcal{F}(u)$ for $\{\emptyset\}$ otherwise.

The following lemma is a mild modification of a lemma of Cummings and Woodin \cite[Lemma 5.1]{M2010}, which shows that every measure sequence $u$ with $\mathrm{lh}(u)<\kappa(u)$ is coherent.

\begin{lemma}\label{p1}
Suppose $u=\langle u(\alpha) \ | \ \alpha<\mathrm{lh}(u) \rangle$ is a measure sequence with $1<\mathrm{lh}(u)<\kappa(u)$. For every $\alpha$ with $0< \alpha<\mathrm{lh}(u)$, let $j_{\alpha}:V \rightarrow N_{\alpha} \cong \mathrm{Ult}(V,u(\alpha))$ be the canonical embedding. Then $u \upharpoonright \alpha \in N_{\alpha}$ and the measure sequence of length $\alpha+1$ given by $j_{\alpha}$ is exactly $u \upharpoonright (\alpha+1)$.
\end{lemma}
\begin{proof}
Let $\kappa:=\kappa(u)$. Since $u$ is a measure sequence, we may find a definable elementary embedding $j:V \rightarrow N$ with $N$ transitive such that
 for every $\alpha$ with $0<\alpha<\mathrm{lh}(u)$,
$u(\alpha)=\{A \subseteq V_{\kappa} \ | \ u \upharpoonright \alpha \in j(A)\}$.

Given any $\alpha$ with $0< \alpha<\mathrm{lh}(u)$. For any $x \in N_{\alpha}$, we can find a function $f:V_{\kappa} \rightarrow V$ representing it, and we denote $x$ by $[f]_{\alpha}$.
Now we will define an embedding $k:N_{\alpha} \rightarrow N$. Let $k([f]_{\alpha})=j(f)(u \upharpoonright \alpha)$ for every $[f]_{\alpha} \in N_{\alpha}$. Then it is easy to see $k$ is well-defined and elementary, and $j=k \circ j_{\alpha}$.

\begin{claim}
$u \upharpoonright \alpha=[\mathrm{id}]_{\alpha} \in N_{\alpha}$.
\end{claim}
\begin{proof}
By the definition of $k$, $k([\mathrm{id}]_{\alpha})=j(\mathrm{id})(u \upharpoonright \alpha)=u \upharpoonright \alpha$. So we only need to prove that $k(u \upharpoonright \alpha)=u \upharpoonright \alpha$.

We will first prove $k(\kappa)=\kappa$.
For every $\beta<\kappa$, $k(\beta)=k(j_{\alpha}(\beta))=j(\beta)=\beta$, so $\mathrm{crit}(k)\geq \kappa$. Meanwhile, $k([\mathrm{id}]_{\alpha}(0))=u(0)=\kappa$
since $k([\mathrm{id}]_{\alpha})=u \upharpoonright \alpha$.
Thus $[\mathrm{id}]_{\alpha}(0)=\kappa$. Consequently, $k(\kappa)=\kappa$ and $\mathrm{crit}(k)>\kappa$.

Now let us prove $k(u\upharpoonright \alpha)=u \upharpoonright \alpha$. It is easy to see $N_{\alpha} \cap V_{\kappa+1}=V_{\kappa+1}= N \cap  V_{\kappa+1}$ and
\begin{equation}\label{e24}
\forall X \in N_{\alpha} \cap V_{\kappa+1}(k(X)=X).
\end{equation}
Take any $\eta<\alpha$. Then $u(\eta)=k''u(\eta) \subseteq k(u(\eta))$ by \eqref{e24}. Since $N_{\alpha} \cap V_{\kappa+1}=V_{\kappa+1}= N \cap  V_{\kappa+1}$, and by the maximality of $u(\eta)$ as a filter, we have $k(u(\eta))=u(\eta)$.
Note also that $k$ fixes the length of $u \upharpoonright \alpha$, we have $k(u\upharpoonright \alpha)=u \upharpoonright \alpha$.
\end{proof}
Now, we prove that the measure sequence of length $\alpha+1$ obtained from $j_{\alpha}$, say $v$, is exactly $u \upharpoonright (\alpha+1)$ by induction on $\beta$ with $\beta \leq \alpha$. Obviously, $v(\beta)=\kappa=u(\beta) \in N_{\alpha}$ if $\beta=0$. Now suppose inductively that $v \upharpoonright \beta=u \upharpoonright \beta \in N_{\alpha}$. For every $X \in V_{\kappa+1}$,
\[
X \in v(\beta) \Leftrightarrow u \upharpoonright \beta=v \upharpoonright \beta \in j_{\alpha}(X) \Leftrightarrow
u \upharpoonright \beta=k(u \upharpoonright \beta) \in k(j_{\alpha}(X))=j(X) \Leftrightarrow X \in u(\beta).
\]
The first and last $``\Leftrightarrow"$ hold by definition, the first equality holds by induction, the second $``\Leftrightarrow "$ holds by elementarity of $k$, and as we proved above, the second equality holds. Hence $v=u \upharpoonright (\alpha+1)$.
\end{proof}

We define next the class $U_{\infty}$ of measure sequences as follows. Let $U_0=\{ u \ | \ u$ is a measure sequence$\}$,
and for every $n<\omega$, let $U_{n+1}=\{ u \in U_n \ | \ U_n \cap V_{\kappa(u)}\in \mathcal{F}(u)\}$. Finally, set $U_{\infty}=\bigcap_{n<\omega}U_n$.
The point is that if $u \in U_{\infty}$, then for every $\alpha$ with $0<\alpha<\mathrm{lh}(u)$, $u(\alpha)$ concentrates on $U_{\infty} \cap V_{\kappa(u)}$.

$U_{\infty}$ can be non-empty when there exists some strong elementary embedding. For instance, if there exists a $j:V \rightarrow M$ with $\mathrm{crit}(j)=\kappa$, $V_{\kappa+2} \subseteq M$ and $M$ is closed under sequences of length $\kappa$, then we can get a measure sequence, say $u$, from $j$ with $\mathrm{lh}(u) \geq (2^{\kappa})^+$ and for every $\alpha<(2^{\kappa})^+$, $u \upharpoonright \alpha \in U_{\infty}$ (see \cite{M2010} or \cite{JW} for details). In particular, the result also holds for every $\alpha$-supercompact embedding with $\alpha \geq |V_{\kappa+2}|$.

In the sequel, if we say $u$ is a measure sequence, we mean that $u$ is in $U_{\infty}$.
Given a measure sequence $u$ of length at least $2$, we may now define the Radin forcing $R_u$.

\begin{definition}\label{rd}
$R_u$ consists of finite sequences $p=\langle (u_0,A_0),\cdots, (u_n,A_n)\rangle$, where
\begin{enumerate}
  \item For every $i \leq n$, $u_i \in U_{\infty}$, $A_i \in \mathcal{F}(u_i)$, and $A_i \subseteq U_{\infty}$.
  \item For every $i<n$, $(u_i,A_i) \in V_{\kappa(u_{i+1})}$.
  \item $u_n=u$.
\end{enumerate}
(We say $u_0,\cdots,u_n$ occur in $p$).

The ordering on $R_u$ is defined as follows. If $p = \langle(u_0, A_0),\cdots,(u_n, A_n)\rangle$ and
$q = \langle(v_0, B_0),\cdots,(v_m, B_m)\rangle$ are in $R_u$, then $p \leq q$ if and only if
\begin{enumerate}
  \item $\{v_0,\cdots,v_m\}\subseteq \{u_0,\cdots,u_n\}$.
  \item For each $j \leq m$ and $i \leq n$, if $v_j = u_i$, then $A_i \subseteq B_j$.
  \item If $i \leq n$ is such that $u_i \notin \{v_0,\cdots,v_m\}$ and if $j \leq m$ is the least such that
$u_i(0) < v_j (0)$, then $u_i \in B_j$ and $A_i \subseteq B_j$.
\end{enumerate}
\end{definition}

Given an $R_u$-generic filter $G$ over $V$, let $g_G:=\langle g_{\alpha} \ | \ \alpha<\mathrm{lh}(g_G) \rangle$
be the \emph{generic sequence given by $G$}. Namely, $g_G$ is the unique sequence consisting of all measure sequence $w$, such that $w \neq u$ and $w$ occurs in some $p\in G$; and if $\alpha<\beta<\mathrm{lh}(g_G)$,
then $\kappa(g_{\alpha})<\kappa(g_{\beta})$. Also let $C_G=\{ \kappa(g_{\alpha}) \ | \ \alpha<\mathrm{lh}(g_G) \}$.
Then $C_G$ is a club subset of $\kappa(u)$. In addition, if $\mathrm{lh}(u)<\kappa(u)$, then there is a condition $p \in R_u$ such that $p$ forces that the order type of $C_G$ is $\omega^{-1+\mathrm{lh}(u)}$. (See \cite{M2010} for details.)

It is not hard to see that $G$ can be recovered from $g_G$, so we may view $g_G$ as the generic object.
Indeed, $G$ consists of all $p \in R_u$ such that
\begin{enumerate}
  \item If $v$ occurs in $p$ and $v \neq u$, then $v=g_{\alpha}$ for some $\alpha <\mathrm{lh}(g_G)$.
  \item For every $\alpha < \mathrm{lh}(g_G)$, $g_{\alpha}$ occurs in some $q \leq p$.
  \end{enumerate}

\begin{definition}\label{dgc}
Suppose $M$ is an inner model of $\mathrm{ZFC}$ and $\delta$ is a limit ordinal. Let $w(\delta)$ be a measure sequence in $M$, and
let $w=\langle w(\alpha) \ | \ \alpha<\delta \rangle$ be a sequence of measure sequences in $M$.
Then $w$ is \emph{geometric} with respect to $w(\delta)$ and $M$ if and only if the following holds:
\begin{enumerate}
  \item\label{single} The sequence $\langle \kappa(w(\alpha)) \ | \ \alpha \leq \delta \rangle$ is increasing continuous.
  \item\label{even} For every limit $\alpha \leq \delta$ and every $A \in M \cap V_{\kappa(w(\alpha))+1}$,
  $A \in \mathcal{F}(w(\alpha))$ if and only if $w \upharpoonright \alpha$ is eventually contained in $A$, i.e.,
  there exists an $\alpha_{A}<\alpha$ such that for every $\gamma$ with $\alpha_A<\gamma<\alpha$, $w(\gamma) \in A$.
\end{enumerate}
\end{definition}

The following theorem, due to W. Mitchell, characterizes Radin generic sequence in terms of the geometric condition. We also follow the notation of
Definition \ref{dgc} in the next theorem.
\begin{theorem}[\cite{M1982}] \label{itu}
A sequence $w$ is geometric with respect to $w(\delta)$ and $M$ if and only if $w$ is a Radin generic sequence given by some $R_{w(\delta)}$-generic filter over $M$.
\end{theorem}
According to \eqref{ev} of Definition \ref{dgc}, for the case $\delta<\kappa(w(0))$, if $w$ is geometric w.r.t. $w(\delta)$ and $M$, i.e., $w$ is a Radin generic sequence given by some $R_{w(\delta)}$-generic filter over $M$, then $\mathcal{F}(w(\alpha))$ concentrates on measure sequences of length less than $\mathrm{lh}(w(\alpha))$ for every $\alpha \leq \delta$. Hence, it is easily seen that
$\mathrm{lh}(w(\alpha))=1$ if $\alpha<\delta$ is a successor ordinal, and $\mathrm{lh}(w(\alpha))=\limsup_{\gamma<\alpha}(\mathrm{lh}(w(\gamma))+1)$ if $\alpha<\delta$ is a limit ordinal. In other words, $\mathrm{lh}(w(\alpha))=\beta_{\alpha}$ for every $\alpha \leq \delta$, where $\beta_{\alpha}$ is defined at the beginning of this section.

Now we may define $u$-iterated ultrapowers as follows.
\begin{definition}\label{it}
Suppose $\mathrm{lh}(u)<\kappa(u)$, and $\delta \leq \omega^{-1+\mathrm{lh}(u)}$ is a limit ordinal.
\begin{enumerate}
  \item $\langle M_{\alpha}, \pi_{\alpha,\alpha'} \ | \ \alpha \leq \alpha' \leq \delta \rangle$ is an \emph{iterated ultrapower} if and only if
\begin{enumerate}
  \item $M_0=V$ and $\pi_{\alpha,\alpha}=\mathrm{id}_{M_{\alpha}}$ for every $\alpha \leq \delta$.
  \item\label{ca3} $M_{\alpha+1} \cong \mathrm{Ult}(M_{\alpha},W_{\alpha})$ is a transitive class, where $W_{\alpha} \in M_{\alpha}$ is a $\kappa_{\alpha}$-complete ultrafilter over $\kappa_{\alpha}$ (or $M_{\alpha}\cap V_{\kappa_{\alpha}}$) for some $\kappa_{\alpha}$, and the ultrapower is constructed
in $M_{\alpha}$; $\pi_{\alpha,\alpha+1}:M_{\alpha} \rightarrow M_{\alpha+1} \cong \mathrm{Ult}(M_{\alpha},W_{\alpha})$ is the canonical embedding, and
for every $\gamma < \alpha$, $\pi_{\gamma,\alpha+1} = \pi_{\alpha,\alpha+1} \circ \pi_{\gamma,\alpha}$.

\item If $\gamma \leq \delta$ is a limit ordinal, then $M_{\gamma}$ is the direct limit of $\langle M_{\alpha}, \pi_{\alpha,\alpha'} \ | \ \alpha \leq \alpha' < \gamma \rangle$,
  and for every $\alpha < \gamma$, $\pi_{\alpha,\gamma}:M_{\alpha} \rightarrow M_{\gamma}$ is the corresponding embedding.
\end{enumerate}
  \item $\langle M_{\alpha}, \pi_{\alpha,\alpha'} \ | \ \alpha \leq \alpha' \leq \delta \rangle$ is a \emph{$u$-iterated ultrapower} if and only if it is an iterated ultrapower, and in \eqref{ca3}, $\kappa_{\alpha}=\pi_{0,\alpha}(\kappa(u))$ and $W_{\alpha}=\pi_{0,\alpha}(u)(\beta_{\alpha})$.
\end{enumerate}
For simplicity of notation, we write $\pi_{\alpha}$  for $\pi_{0,\alpha}$ for every $\alpha\leq \delta$, $\pi$ for $\pi_{\delta}$ and $M$ for $M_{\delta}$. Here we require that the length $\delta$ of a $u$-iterated ultrapower is less than or equal to
$\omega^{-1+\mathrm{lh}(u)}$, because $\beta_{\alpha}$ should be less than $\mathrm{lh}(u)$ for every $\alpha<\delta$.
\end{definition}
Next, we follow the notation of the definition above. Let $\langle M_{\alpha}, \pi_{\alpha,\alpha'} \ | \ \alpha \leq \alpha' \leq \delta \rangle$ be the $u$-iterated ultrapower of length $\delta$, let $w=\langle \pi_{\alpha}(u) \upharpoonright \beta_{\alpha} \ | \ \alpha<\delta \rangle$ and let $w(\delta)=\pi_{\delta}(u) \upharpoonright \beta_{\delta}$.

Now we will use $u$-iterated ultrapower to construct a Radin generic sequence over some target model. We first prove the following lemma.

\begin{lemma}\label{c5}
Suppose $\theta \leq \delta$ is a limit ordinal. If $\eta<\theta$ satisfies that
$\beta_{\alpha}<\beta_{\theta}$ for every $\alpha$ with $\eta \leq \alpha<\theta$, then for every $\bar{A} \in \mathcal{F}(\pi_{\eta}(u) \upharpoonright \beta_{\theta}) $, we have
\begin{equation}\label{q10}
\{ w(\alpha) \ | \ \eta \leq \alpha <\theta \} \subseteq \pi_{\eta,\theta}(\bar{A}).
\end{equation}
In particular, for every limit $\theta \leq \delta$, if $A \in \mathcal{F}(\pi_{\theta}(u) \upharpoonright \beta_{\theta})=\mathcal{F}(w(\theta))$, then
\begin{equation}\label{q4}
w \upharpoonright \theta \text{ is eventually contained in } A.
\end{equation}
\end{lemma}

\begin{proof}
For every $\alpha$ with $\eta \leq \alpha<\theta$, since $\bar{A} \in \mathcal{F}(\pi_{\eta}(u) \upharpoonright \beta_{\theta})$ and $\pi_{\eta,\alpha}$ is elementary, we have $M_{\alpha} \models  \pi_{\eta,\alpha}(\bar{A})\in \mathcal{F}(\pi_{\alpha}(u) \upharpoonright \beta_{\theta})$.
By our assumption, $\beta_{\alpha}<\beta_{\theta}$, so
$M_{\alpha} \models \pi_{\eta,\alpha}(\bar{A})\in \pi_{\alpha}(u)(\beta_{\alpha})$. Meanwhile,
by the definition of the $u$-iterated ultrapower $\langle M_{\alpha'}, \pi_{\alpha',\alpha''} \ | \ \alpha' \leq \alpha'' \leq \delta \rangle$,
we have $\pi_{\alpha,\alpha+1}:M_{\alpha} \rightarrow M_{\alpha+1} \cong \mathrm{Ult}(M_{\alpha},\pi_{\alpha}(u)(\beta_{\alpha}))$. Note also that $\pi_{\alpha}(u)$ is a measure sequence in $M_{\alpha}$, we have
the measure sequence of length $\beta_{\alpha}+1$ obtained from $\pi_{\alpha,\alpha+1}$ is exactly $\pi_{\alpha}(u)\upharpoonright (\beta_{\alpha}+1)$ by Lemma \ref{p1}. Hence,
\[
M_{\alpha+1} \models w(\alpha)=\pi_{\alpha}(u)\upharpoonright \beta_{\alpha} \in \pi_{\alpha,\alpha+1}(\pi_{\eta,\alpha}(\bar{A}))=\pi_{\eta,\alpha+1}(\bar{A}).
\]
Since $\pi_{\alpha+1,\theta}$ is elementary, and $w(\alpha)$ is fixed by $\pi_{\alpha+1,\theta}$, i.e., $\pi_{\alpha+1,\theta}(w(\alpha))=w(\alpha)$, we have
$M_{\theta} \models w(\alpha) \in \pi_{\eta,\theta}(\bar{A})$.
Hence $w(\alpha) \in \pi_{\eta,\theta}(\bar{A})$. So \eqref{q10} holds.

Now take any limit $\theta \leq \delta$, we prove that \eqref{q4} holds. Since $A \in \mathcal{F}(\pi_{\theta}(u) \upharpoonright \beta_{\theta})$,
we can pick a sufficient large $\bar{\theta}<\theta$, so that $\beta_{\alpha}<\beta_{\theta}$
for every $\alpha$ with $\bar{\theta} \leq \alpha<\theta$, and
 there exists an $\bar{A} \in M_{\bar{\theta}}$ such that $\pi_{\bar{\theta},\theta}(\bar{A})=A$. Then $\bar{A} \in \mathcal{F}(\pi_{\bar{\theta}}(u) \upharpoonright \beta_{\theta})$. Hence $\{ w(\alpha) \ | \ \bar{\theta} \leq \alpha<\theta \} \subseteq A$,
which means $w \upharpoonright \alpha$ is eventually contained in $A$.
\end{proof}
The point of the lemma above is that by Lemma \ref{p1}, for every $\alpha < \delta$, the measure
sequence of length $\beta_{\alpha}+1$ obtained from $\pi_{\alpha,\alpha+1}$ is exactly
$\pi_{\alpha}(u)\upharpoonright (\beta_{\alpha}+1)$. So for any iterated ultrapower $\langle M_{\alpha}, \pi_{\alpha,\alpha'} \ | \ \alpha \leq \alpha' \leq \delta \rangle$, if the measure sequence of length $\beta_{\alpha}$ + 1 obtained from $\pi_{\alpha,\alpha+1}$ is $\pi_{\alpha}(u) \upharpoonright (\beta_{\alpha} + 1)$ for every $\alpha<\delta$, then the lemma above also holds. For example, we may let $\pi_{\alpha,\alpha+1}$ be the ultrapower map given by $\pi_{\alpha}(u)(\theta)$ for every $\alpha < \delta$ in
\eqref{ca3} if there is a $\theta$ with $\delta < \theta < \mathrm{lh}(u)$. Then the lemma above also holds for this new iterated ultrapower.

The following theorem is essentially due to Radin \cite{R1982}, see also \cite[Theorem 6.7.1]{JW}.
\begin{theorem}[\cite{R1982}]\label{thm1}
The sequence $w$ is geometric with respect to $w(\delta)$ and $M$(=$M_{\delta}$), and we have an $R_{w(\delta)}$-generic filter over $M$ given by $w$.
\end{theorem}
\begin{proof}
We only prove \eqref{even} of the geometric condition here, i.e., for every limit $\alpha \leq \delta$ and every $A \in M \cap V_{\kappa(w(\alpha))+1}$,
$A \in \mathcal{F}(w(\alpha))$ if and only if $w \upharpoonright \alpha \text{ is eventually contained in } A$.

If $A \in \mathcal{F}(w(\alpha))$,
then by Lemma \ref{c5}, $w \upharpoonright \alpha$ is eventually contained in $A$.

If $A \notin \mathcal{F}(w(\alpha))$, then $A \notin \pi_{\alpha}(u)(\gamma)$ for some $\gamma<\beta_{\alpha}$.
Let $B=M \cap V_{\kappa(w(\alpha))} \setminus \{ x \in A \ | \ \mathrm{lh}(x)=\gamma \}$, then $B \in \mathcal{F}(w(\alpha))$. Hence, $w \upharpoonright \alpha$ is eventually contained in $B$. Note also that $\langle \eta<\alpha \ | \ \beta_{\eta}=\gamma \rangle$ is unbounded in $\alpha$, we have $w \upharpoonright \alpha$ is not eventually contained in $A$.

Therefore, $w$ is geometric w.r.t. $w(\delta)$ and $M$, and we have a $R_{w(\delta)}$-generic filter over $M$ obtained from $w$.
\end{proof}
\section{main idea of the consistency result (theorem \ref{t1})} \label{sec3}
Suppose $\kappa$ is a supercompact cardinal and $\delta<\kappa$ is a $\lambda$-measurable cardinal. Let $j:V \rightarrow M$ be a suitable supercompact ultrapower map, let $i:M \rightarrow N$ be an ultrapower map given by some $\lambda$-complete uniform ultrafilter in $M$, and let $\pi=i \circ j$. Let $u$ be the measure sequence of length $\delta$ obtained from $j$, and let $G$ be a suitable $R_u$-generic filter over $V$.

For the purpose of making this paper easier to read, we next give the idea of the proof of Bagaria-Magidor from \cite[Theorem 6.1]{BM2014}, and our idea of the proof of Theorem \ref{t1}.

In the proof of Bagaria-Magidor, they only consider the case $\lambda=\delta$, i.e., $\delta$ is measurable, and take the Radin forcing $R_u$ to turn $\kappa$ into a $\lambda$-strongly compact cardinal.

To prove the $\lambda$-strong compactness of $\kappa$, they lift the composite embedding $\pi=i \circ j$ in a $\pi(R_u) \verb|\| R_{i(u) \upharpoonright \delta}$-generic extension of $V[G]$. Here, $j$ grants the $\delta$-strong compactness of $\kappa$ at the end, and $i$ makes $\langle (i(u) \upharpoonright \delta,i(A)) \rangle$ addible to $\pi(\langle (u,A)\rangle)$ for every $\langle (u,A)\rangle \in G$. Namely,
$\langle (i(u) \upharpoonright \delta,i(A)),(\pi(u),\pi(A) \rangle \in \pi(R_u)$ for every $\langle (u,A)\rangle \in G$. In addition, since $R_u$ has a particular closure property, $i''g_G$ can generate an $R_{i(u)\upharpoonright \delta}$-generic filter by a variation of the transfer argument (see \cite{BM2014}, also \cite[15]{J2010}). Thus by Silver's criterion (see \cite[Proposition 9.1]{J2010}), a lifting embedding of $\pi$ can be obtained.
However, this embedding is not definable in $V[G]$. To remedy this, Bagaria-Magidor use a closure argument, which relies not only on the closure of the Radin forcing itself, but also on the closure of $N$, to show that the filter generated by the lifted embedding
is $\lambda$-complete. Thus $\kappa$ is $\lambda$-strongly compact in $V[G]$ (here and next, actually $\kappa$ is $\lambda$-strongly compact up to $\kappa'$ for some $\kappa'$, but a simple trick can solve this problem by lifting class-many embeddings with the same $u$).

In our proof, we also take the Radin forcing $R_u$ to turn $\kappa$ into a $\lambda$-strongly compact cardinal, but handle the general case (i.e., $\lambda$ may not equal to $\delta$).
The major difference, or the novelty of the proof, is the lifting argument. The argument is more complicated in the general case, because $\delta$ is $\lambda$-measurable, and so it may have stronger consistency strength than measurability (however, we don’t know that under the existence of a supercompact cardinal, whether it is possible to make some regular cardinal, for example, $\lambda^+$, a $\lambda$-measurable cardinal). We next give more details about the lifting argument.

We start with the strategy of Bagaria-Magidor, and we still consider to lift the composite embedding $\pi$ by using Silver's criterion.
But there is a problem. There may exist unboundedly many $\alpha<\delta$ with $\sup(i''\alpha)<i(\alpha)$ below $\delta$, because $\delta$ may surpass the critical point of $\pi$. Then for every such $\alpha$, there will be a gap below $\pi(g_{\alpha})$ to be filled in, namely, a generic
object in a similar sense of $\pi(R_u) \verb|\| R_{i(u) \upharpoonright \delta}$.
But there is no place to fill in the gap below $\pi(g_{\alpha})$,
since there are unboundedly many $\pi(g_{\gamma})$ below $\pi(g_{\alpha})$.
Hence, we can't lift the embedding $\pi$.

To overcome the problem, we invoke Theorem \ref{thm1}, which states that a Radin generic object can be generated by some iterated ultrapower, to fill in a gap. However, we may need to fill in many gaps. Hence, a sequence of iterated ultrapowers should be taken to fill in all these gaps (an iterated ultrapower above $\kappa$ is also taken to fill the gap above $\kappa$, i.e., the counterpart of $\pi(R_u) \verb|\| R_{i(u) \upharpoonright \delta}$, so that we can avoid the closure argument).

So we may get an iteration, also say $\pi$\footnote{We exchanged the order of the supercompact embedding and the iteration below $\kappa$ in $\pi$ in the proof of Theorem \ref{t1}, so that the Radin sequence generated by $\pi''g_G$ and the iteration has a uniform definition.}, which lives in $V[G]$ since it is guided by the $R_u$-generic object $G$ over $V$. And we can build a $\pi(R_{u})$-generic object $H$ over the target model in $V[G]$. In addition, $\pi''G \subseteq H$. So we may obtain a lifting embedding of $\pi$ in $V[G]$ by using Silver's criterion, which witnesses the $\lambda$-strong compactness of $\kappa$.

\section{Main Results} \label{sec4}

The following well-known fact shows that Radin generic sequence above an $\omega_1$-strongly compact cardinal destroys $\omega_1$-strong compactness of the smaller cardinal.

\begin{fact}[Folklore]\label{fact41}
Suppose $u$ is a measure sequence of length at least $2$. Then in a $R_u$-generic extension of $V$, there is no $\omega_1$-strongly compact cardinal below $\kappa(u)$.
\end{fact}
\begin{proof}
Let $\kappa:=\kappa(u)$ and let $G$ be an $R_u$-generic filter. Suppose, assuming a contradiction, that there is a $\gamma<\kappa$ such that $\gamma$ is $\omega_1$-strongly compact.
Note that there is a Prikry sequence contained in $C_G \setminus \gamma$. Then there is a $\square_{\theta,\omega}$-sequence, say $\vec{\mathcal{C}}=\langle \mathcal{C}_{\alpha} \ | \ \alpha<\theta^+ \rangle$, for some $\theta \in C_G \setminus \gamma$ (see \cite[Theorem 4.2]{CS2002}).
By the $\omega_1$-strong compactness of $\gamma$, there exists an elementary embedding $k:V[G] \rightarrow M'$ such that $\sup(k''\theta^+)<k(\theta^+)$. Then $k(\vec{C})$ is a
$\square_{k(\theta),\omega}$-sequence in $M'$. Let $\beta :=\sup(k''\theta^+)$, and pick a $C' \in k(\vec{\mathcal{C}})(\beta)$. Then we have
\begin{enumerate}[label=\upshape(\roman*), leftmargin=*, widest=iii]
    \item $C'$ is a club subset of $\beta$.
    \item\label{item2} $C'$ has order type at most $k(\theta)$.
    \item\label{item3} If $\alpha$ is a limit point of $C'$, then $C' \cap \alpha \in \mathcal{C}_{\alpha}$.
\end{enumerate}

It is easy to see that $k''\theta^+$ is a stationary subset of $\beta$, so there are unboundedly many $\alpha<\theta^+$, such that $k(\alpha)$ is a limit point of $C'$. For any such $\alpha$, by \ref{item3}, we have $C' \cap k(\alpha)=k(C'_{\alpha})$ for some $C'_{\alpha} \in \mathcal{C}_{\alpha}$. Hence, these $C'_{\alpha}$ are pairwise compatible, i.e., for any such $\alpha<\alpha'$, $k(C'_{\alpha'})\cap k(\alpha)=k(C'_{\alpha})$. By elementarity of $k$, we have $C'_{\alpha'}\cap \alpha=C'_{\alpha}$. So the union of these $C'_{\alpha}$, say $C$, is a club subset of $\theta^+$, and $C \cap \alpha=C'_{\alpha}$ for any such $\alpha$. Hence, $C$ has order type $\theta^+$. However, any such $C'_{\alpha}$ has order type at most $\theta$, a contradiction. Hence, there is no $\omega_1$-strongly compact cardinal below $\kappa$.
\end{proof}

\begin{theorem}\label{t1}
Suppose $\kappa$ is a supercompact cardinal,
 and $\delta <\kappa$ is a $\lambda$-measurable cardinal for some uncountable cardinal $\lambda$.
Then in a Radin generic extension of $V$ that preserves the $\lambda$-measurability of $\delta$,
$\kappa$ is the least $\lambda$-strongly compact cardinal and has cofinality $\delta$.
\end{theorem}

\begin{proof}
First let us find a measure sequence $u$ on $\kappa$ with length $\delta$  to obtain a suitable Radin forcing $R_u$.

For every $\kappa'>|V_{\kappa+2}|$, let $U_{\kappa'}$ be a normal fine measure over $\mathcal{P}_{\kappa}(\kappa')$, and let
 $j_{\kappa'}:V \rightarrow M_{\kappa'} \cong \mathrm{Ult}(V,U_{\kappa'})$ be the corresponding supercompact embedding. Let $u_{\kappa'}$ be the measure sequence
  of length $\delta$ obtained from $j_{\kappa'}$. Since there are at most $2^{2^{\kappa}}$ many such measure sequences of length $\delta$,
 there exists a proper class $\mathcal{S}$ of ordinals and a measure sequence $u=\langle u(\alpha) \ | \ \alpha<\delta \rangle$
   such that for every $\kappa' \in \mathcal{S}$,  $u=u_{\kappa'}$.

Let $R_u$ be the Radin forcing for $u$. Pick a condition $\langle (u,C) \rangle \in R_u$
so that $C \cap V_{\delta+1}=\emptyset$ and it forces $\mathrm{lth}(g_{\dot{G}})=\delta$.
Then $C$ consists of measure sequences of length less than $\delta$. Let $G$ be an $R_u$-generic filter over $V$ with $\langle (u,C) \rangle \in G$. Then it adds no new subsets of $\delta$. So $\delta$ is also $\lambda$-measurable in $V[G]$. Let $g_G=\langle g_{\alpha} \ | \ \alpha<\delta \rangle$ be the generic sequence given by $G$. Then $g_G$ is geometric w.r.t. $u$ and $V$ by Theorem \ref{itu}.

Now we will define a composite embedding $\pi$.
Since $\delta$ is $\lambda$-measurable, by Proposition \ref{pu}, there exists a $\lambda$-complete uniform ultrafilter over $\delta$, say $W$. Let
$i:V \rightarrow N_0=\mathrm{Ult}(V,W)$
be the canonical map. W.l.o.g., we may assume $\mathrm{crit}(i)=\lambda$. Otherwise, let $\lambda'=\mathrm{crit}(i)$. Then we can prove that $\kappa$ is $\lambda'$-strongly compact with the same proof. Notice that as $\lambda' \geq \lambda$, $\kappa$ is also $\lambda$-strongly compact.

For simplicity of notation, let $g_{\delta}:=u$. For every $\alpha \leq \delta$, define
\begin{equation}\label{e28}
s(\alpha)=\begin{cases}
            \sup(i''\alpha), & \text{if $\alpha$ is a limit ordinal}, \\
            i(\alpha), & \text{otherwise}.
          \end{cases}
\end{equation}
Then the intervals $[s(\alpha),i(\alpha)]$ for $\alpha \leq \delta$ constitute a partition of $i(\delta)+1$ by \eqref{e28}. So for every $\theta \leq i(\delta)$, we may let $[\theta]$ be the
unique ordinal such that $s([\theta])\leq \theta \leq i([\theta])$.

Take any $\kappa' \in \mathcal{S}$ and let $U:=U_{\kappa'}$.
Then in $V[G]$, let us construct an iterated ultrapower $\langle N_{\theta},\pi_{\theta,\theta'} \ | \ \theta \leq \theta'\leq i(\delta) \rangle$ as follows:
\begin{enumerate}[leftmargin=*, widest=iii]
  \item $\pi_{\theta,\theta}=\mathrm{id}_{N_{\theta}}$ for every $\theta \leq i(\delta)$.
  \item If $\theta \leq i(\delta)$ is a limit ordinal, then
$N_{\theta}$ is the direct limit of $\langle N_{\theta_0},\pi_{\theta_0,\theta_1} \ | \ \theta_0 \leq \theta_1\leq \theta \rangle$, together with elementary embeddings
$\pi_{\theta_0,\theta}:N_{\theta_0} \rightarrow N_{\theta}$ for all $\theta_0 < \theta$.
  \item If $\theta \leq i(\delta)$ is a successor ordinal with $s([\theta]) \leq \theta < i([\theta])$, then $N_{\theta+1}$ is the transitive class isomorphic to $\mathrm{Ult}(N_{\theta},\pi_{\theta}(U))$ if $\theta=s(\delta)$, or isomorphic to $\mathrm{Ult}(N_{\theta},\pi_{\theta}(g_{[\theta]})(\beta_{\theta}))$, otherwise; $\pi_{\theta,\theta+1}:N_{\theta} \rightarrow N_{\theta+1}$ is the corresponding ultrapower map, and for every $\gamma<\theta$, let
   $\pi_{\gamma,\theta+1}=\pi_{\theta,\theta+1} \circ \pi_{\gamma,\theta}$.
  \item\label{n4} If $\theta \leq i(\delta)$ is a successor ordinal with $\theta=i([\theta])$, then $N_{\theta+1}=N_{\theta}$; $\pi_{\theta,\theta+1}=\mathrm{id}_{N_{\theta}}$, and for every $\gamma<\theta$, let $\pi_{\gamma,\theta+1}=\pi_{\gamma,\theta}$.
\end{enumerate}
For simplicity of notation, let $\pi_{\theta}=\pi_{0,\theta} \circ i$ for every $\theta \leq i(\delta)$, let $\pi=\pi_{i(\delta)}$ and $N=N_{i(\delta)}$.
Let $w=\langle \pi_{\theta}(g_{[\theta]})\upharpoonright \beta_{\theta}\ | \ \theta<i(\delta) \rangle$, let $w(i(\delta))=\pi_{i(\delta)}(g_{\delta})\upharpoonright \beta_{i(\delta)}=\pi(g_{\delta})$, and let $\vec{\kappa}=\langle \kappa(w(\theta)) \ | \ \theta \leq i(\delta) \rangle$. Then $\vec{\kappa}=\langle \pi_{\theta}(\kappa(g_{[\theta]})) \ | \ \theta \leq i(\delta) \rangle$ since $\kappa(w(\theta))=\kappa(\pi_{\theta}(g_{[\theta]}))=\pi_{\theta}(\kappa(g_{[\theta]}))$ for every $\theta \leq i(\delta)$.

In the iteration $\pi$, these identity class function in \eqref{n4} are used for simplicity of notation.

The iterated ultrapower $\langle N_{\theta},\pi_{\theta,\theta'} \ | \ \theta \leq \theta' \leq i(\delta) \rangle$ is well-founded (see \cite[Theorem 19.30]{J2003}), so $N_{\theta}$ is transitive for every $\theta \leq i(\delta)$.

\begin{claim}\label{cg}
In $V[G]$, $w$ is geometric with respect to $w(i(\delta))$ and $N$. That is,
\begin{enumerate}
  \item\label{club} The sequence $\vec{\kappa}$ is increasing continuous.
  \item\label{ev} For every limit $\theta \leq i(\delta)$ and every $A \in N \cap V[G]_{\kappa(w(\theta))+1}$,
  $A \in \mathcal{F}(w(\theta))$ if and only if $w \upharpoonright \theta$ is eventually contained in $A$.
\end{enumerate}
\end{claim}
\begin{proof}
For every $\alpha \leq \delta$ with $s(\alpha)<i(\alpha)$, note that $\langle N_{\theta},\pi_{\theta,\theta'} \ | \ s(\alpha) < \theta \leq \theta' \leq i(\alpha) \rangle$ is the $\pi_{s(\alpha)+1}(g_{\alpha})$-iterated ultrapower
of length $i(\alpha)$ over $N_{s(\alpha)+1}$, we have $w \upharpoonright (s(\alpha),i(\alpha)]$ is geometric
w.r.t. $w(i(\alpha))$ and $N_{i(\alpha)}$ by Theorem \ref{thm1}. Note also that $N_{i(\alpha)}\cap V[G]_{\kappa(w(i(\alpha)))+1}= N \cap V[G]_{\kappa(w(i(\alpha)))+1}$ and $w(i(\alpha)) \in N$, it follows that $w \upharpoonright (s(\alpha),i(\alpha)]$ is also geometric w.r.t.
$w(i(\alpha))$ and $N$. Hence for the proof of \eqref{club} and \eqref{ev}, we only need to consider the case that
$\theta \leq i(\delta)$ is a limit ordinal and $\theta = s([\theta])$.

\begin{fact}
For every $\alpha \leq \delta$, if $\gamma \leq s(\alpha)$, then $\kappa(g_{\alpha})$ is fixed by $\pi_{\gamma}$, i.e., $\pi_{\gamma}(\kappa(g_{\alpha}))=\kappa(g_{\alpha})$.
\end{fact}
\begin{proof}
If $\gamma<s(\alpha)$, then the iterated ultrapower $\langle N_{\theta},\pi_{\theta,\theta'} \ | \ s(\alpha) < \theta \leq \theta' \leq \gamma \rangle$ is taken in $V[w \upharpoonright [\gamma]]$. Since $\kappa(g_{\alpha})>\kappa(\pi_{\gamma}(g_{[\gamma]}))$ is inaccessible in $V[w \upharpoonright [\gamma]]$, it is fixed by $\pi_{\gamma}$.

If $\gamma =s(\alpha)$, then $\gamma=\lim_{\xi<\alpha}i(\xi)$. So for every $\eta<\pi_{\gamma}(\kappa(g_{\alpha}))$, there is a  $\xi<\alpha$ and an $\bar{\eta}<\pi_{i(\xi)}(\kappa(g_{\alpha}))$ such that $\pi_{i(\xi),\gamma}(\bar{\eta})=\eta$. Meanwhile, since $C_G$ is a club subset of $\kappa$ and $\pi_{i(\xi)}(\kappa(g_{\alpha'}))=\kappa(g_{\alpha'})$ for every $\alpha'$ with $\xi<\alpha'\leq \alpha$, we have $\{\pi_{i(\xi)}(\kappa(g_{\alpha'})) \ | \ \xi<\alpha'<\alpha \}=\{\kappa(g_{\alpha'}) \ | \ \xi<\alpha'<\alpha \}$ is also a club subset of $\pi_{i(\xi)}(\kappa(g_{\alpha}))=\kappa(g_{\alpha})$.
Hence, $\bar{\eta}<\pi_{i(\xi)}(\kappa(g_{\alpha'}))$ for some $\alpha'<\alpha$. Then by elementarity of $\pi_{i(\xi),\gamma}$, we have $\eta=\pi_{i(\xi),\gamma}(\bar{\eta})<\pi_{\gamma}(\kappa(g_{\alpha'}))=
\pi_{i(\alpha')}(\kappa(g_{\alpha'}))<\pi_{i(\alpha')}(\kappa(g_{\alpha}))=\kappa(g_{\alpha})$.
Hence, $\pi_{\gamma}(\kappa(g_{\alpha}))=\kappa(g_{\alpha})$.
\end{proof}

By the fact above, the sequence $\langle \kappa(w(s(\alpha))) \ | \ \alpha \leq \delta \rangle=\langle \kappa(g_{\alpha})) \ | \ \alpha \leq \delta \rangle$ is increasing continuous. Meanwhile, for every $\alpha < \delta$ with $s(\alpha)<i(\alpha)$, $\vec{\kappa} \upharpoonright (s(\alpha),i(\alpha)]$ is between $w(s(\alpha))=\kappa(g_{\alpha})$ and $w(s(\alpha+1))=w(i(\alpha)+1)=\kappa(g_{\alpha+1})$, and $\vec{\kappa} \upharpoonright (s(\delta),i(\delta)]$ is above $\kappa(g_{\delta})$. Notice also that $\vec{\kappa} \upharpoonright (s(\alpha),i(\alpha)]$ is increasing continuous, it follows that \eqref{club} holds.

Take any limit $\theta \leq i(\delta)$ with $\theta=s([\theta])$.
\begin{lemma}\label{le10}
For every $\eta \leq \theta$, if $B \in \mathcal{F}(\pi_{\eta}(g_{[\theta]}))$, then there is an $m<[\theta]$ such that
\begin{enumerate}[label=\upshape(\roman*), leftmargin=*, widest=iii]
  \item\label{e51} For every $\alpha$ with $\ m \leq \alpha<[\theta]$, we have $\pi_{\eta}(g_{\alpha}) \in B$;
  \item\label{e52} For every limit $\alpha$ with $m<\alpha<[\theta]$, we have $B \cap V[G]_{\pi_{\eta}(\kappa(g_{\alpha}))} \in \mathcal{F}(\pi_{\eta}(g_{\alpha}))$.
\end{enumerate}
\end{lemma}

\begin{proof}
We prove \ref{e51} and \ref{e52} by induction on $\eta$.  If $\eta = 0$, then since $\pi_{\eta} = i : V \rightarrow N_0 \cong
\mathrm{Ult}(V,W)$, $|W| < \kappa(g_0)$, like in \emph{Case} 2. below, we easily know that there is an $m$ such
that \ref{e51} and \ref{e52} hold for $B$. Now suppose \ref{e51} and \ref{e52} hold for every $\xi<\eta$. Then there are
two cases for $\eta$ (the case that $\eta$ is a successor ordinal with $\eta-1 = i([\eta-1])$ is trivial since
$\pi_{\eta-1}= \mathrm{id}$, so we omit it):

\emph{Case 1.} $\eta$ is a limit ordinal. Then there is an $\bar{\eta}<\eta$ and a $\bar{B} \in N_{\bar{\eta}}$ such that $\pi_{\bar{\eta},\eta}(\bar{B})=B$. Since $B \in \mathcal{F}(\pi_{\eta}(g_{[\theta]}))$ and $\pi_{\bar{\eta},\eta}$ is elementary, we have $\bar{B} \in \mathcal{F}(\pi_{\bar{\eta}}(g_{[\theta]}))$. Thus
by induction, there is an $m<[\theta]$ such that \ref{e51} and \ref{e52} hold for $\bar{\eta}$ and $\bar{B}$.
Since $\pi_{\bar{\eta},\eta}$ is elementary, it follows that for such an $m$, \ref{e51} and \ref{e52} hold for $\eta$ and $B$.

\emph{Case 2.} $\eta$ is a successor ordinal with $s([\eta-1]) \leq \eta-1<i([\eta-1])$.
Then $\eta <\theta$, $[\eta]<[\theta]$
and $\pi_{\eta-1,\eta}$ is the ultrapower map given by $\pi_{\eta-1}(g_{[\eta-1]})(\beta_{\eta-1})$.
Since $B \in \mathcal{F}(\pi_{\eta}(g_{[\theta]}))$ is in $N_{\eta}$, it can be represented by a function $f \in N_{\eta-1}$ with domain
$N_{\eta-1} \cap V[G]_{\pi_{\eta-1}(\kappa(g_{[\eta-1]}))}$,
and for every $x \in \mathrm{dom}(f)$, $f(x) \in \mathcal{F}(\pi_{\eta-1}(g_{[\theta]}))$.
Now work in $N_{\eta-1}$, and let $\bar{B}=\bigcap_{x \in \mathrm{dom}(f)}f(x)$. Then $\bar{B} \in \mathcal{F}(\pi_{\eta-1}(g_{[\theta]}))$ since $\mathcal{F}(\pi_{\eta-1}(g_{[\theta]}))$ is $\pi_{\eta-1}(\kappa(g_{[\theta]}))$-complete and $\pi_{\eta-1}(\kappa(g_{[\theta]}))>\pi_{\eta-1}(\kappa(g_{[\eta-1]}))$,
By induction, there is an $m<[\theta]$ such that \ref{e51} and \ref{e52} hold for $\eta-1$ and $\bar{B}$.

Meanwhile, $\pi_{\eta-1,\eta}(\bar{B})\subseteq B$ since $\bar{B} \subseteq f(x)$ for each $x \in \mathrm{dom}(f)$. Hence,
since $\pi_{\eta-1,\eta}$ is elementary, it follows that for such an $m$, \ref{e51} and \ref{e52} hold for $\eta$ and $B$ as well.

So in any case, \ref{e51} and \ref{e52} hold. Hence by induction, the lemma holds.
\end{proof}

Now we will prove that \eqref{ev} holds. If $A \in \mathcal{F}(w(\theta))=\mathcal{F}(\pi_{\theta}(g_{[\theta]})\upharpoonright \beta_{\theta})$, then
w.l.o.g., we may assume $A \in \mathcal{F}(\pi_{\theta}(g_{[\theta]}))$
since for every sufficient large $\xi<\theta$, the length of $w(\xi)$ is less than $\beta_{\theta}$.
Since $\theta$ is a limit ordinal, there is a $\bar{\theta}<\theta$ and an $\bar{A} \in N_{\bar{\theta}}$
such that $\pi_{\bar{\theta},\theta}(\bar{A})=A$. Note that $A \in \mathcal{F}(\pi_{\theta}(g_{[\theta]}))$ and $\pi_{\bar{\theta},\theta}$ is elementary, we have $\bar{A} \in \mathcal{F}(\pi_{\bar{\theta}}(g_{[\theta]}))$.

Now by Lemma \ref{le10}, there is an $m$ with $[\bar{\theta}]<m<[\theta]$ such that the following holds:
\begin{enumerate}[label=\upshape(\roman*), leftmargin=*, widest=iii]
  \item\label{eq61} For every $\alpha$ with $m \leq \alpha<[\theta]$, $\pi_{\bar{\theta}}(g_{\alpha}) \in \bar{A}$;
  \item\label{eq62} For every limit $\alpha$ with $m<\alpha<[\theta]$, $\bar{A} \cap V[G]_{\kappa(\pi_{\bar{\theta}}(g_{\alpha}))} \in \mathcal{F}(\pi_{\bar{\theta}}(g_{\alpha}))$.
\end{enumerate}
For every $\alpha$ with $m \leq \alpha<[\theta]$, since \ref{eq61} holds and $\pi_{\bar{\theta},\theta}$ is elementary, it follows that $w(i(\alpha))=\pi_{i(\alpha)}(g_{\alpha})=\pi_{\theta}(g_{\alpha}) \in \pi_{\bar{\theta},\theta}(\bar{A})=A$. Hence,
\begin{equation}\label{equ3}
\{ w(i(\alpha)) \ | \ m \leq \alpha<[\theta] \} \subseteq A.
\end{equation}
For every limit $\alpha$ with $m < \alpha< [\theta]$, since \ref{eq62} holds and $\pi_{\bar{\theta},s(\alpha)}$ is elementary, we have
\[
\pi_{\bar{\theta},s(\alpha)}(\bar{A}) \cap V[G]_{\kappa(\pi_{s(\alpha)}(g_{\alpha}))} \in \mathcal{F}(\pi_{s(\alpha)}(g_{\alpha})).
\]
If $s(\alpha)<i(\alpha)$, then $\beta_{\eta}<\beta_{i(\alpha)}$ for every $\eta$ with $s(\alpha) \leq \eta<i(\alpha)$. Hence by Lemma \ref{c5},
\begin{equation}\label{e34}
\{ w(\eta) \ | \ s(\alpha) \leq \eta < i(\alpha) \} \subseteq \pi_{\bar{\theta},i(\alpha)}(\bar{A}) \cap V[G]_{\pi_{i(\alpha)}(\kappa(g_{\alpha}))} \subseteq A.
\end{equation}
Therefore, we have $\{ w(\eta) \ | \ i(m) < \eta <\theta \} \subseteq A$ by (\ref{equ3}) and (\ref{e34}) for some $m<[\theta]$. In other words,
 $w \upharpoonright \theta$ is eventually contained in $A$.

If $A \notin \mathcal{F}(w(\theta))$, then $A \notin w(\theta)(\eta)$ for some $0<\eta<\beta_{\theta}$. Hence $\{ \gamma<\theta \ | \ w(\gamma) \notin A$ and
$\beta_{\gamma}=\eta \}$ is unbounded in $\theta$, which means that $w \upharpoonright \theta$ is not eventually contained in $A$.
\end{proof}

In $V[G]$, let $H\subseteq \pi(R_{u})$ be the filter given by $w$. Namely, $H$ is the set of all $p \in \pi(R_{u})$ such that
\begin{enumerate}[label=\upshape(\roman*), leftmargin=*, widest=iii]
  \item If $v$ occurs in $p$, then $v=w(\theta)$ for some $\theta \leq i(\delta)$;
  \item For any $\theta \leq i(\delta)$, $w(\theta)$ occurs in some $q \leq p$.
\end{enumerate}

It follows from Claim \ref{cg} that $H$ is $\pi(R_u)$-generic over $N$.
Now we will prove that $\pi '' G \subseteq H$.
Take any condition $p=\langle (g_{\alpha_1},A_1),...,(g_{\alpha_n},A_n) \rangle \in G$, where $\alpha_n=\delta$. We need to prove that $\pi(p) \in H$. Let $\alpha_0=-1$ for simplicity. Note that $\pi(g_{\alpha_j})=\pi_{i(\alpha_j)}(g_{\alpha_j})=w(i(\alpha_j))$ and $\pi(A_j)=\pi_{i(\alpha_j)}(A_j)$ for every $1 \leq j \leq n$, so we have
\[
\pi(p)=\langle (w(i(\alpha_1)),\pi_{i(\alpha_1)}(A_1)),\cdots,(w(i(\alpha_n)),\pi_{i(\alpha_n)}(A_n)) \rangle.
\]
Then by the definition of $H$, we only need to prove that for every $m \leq n$, $w(\theta) \in \pi_{i(\alpha_{m+1})}(A_{m+1})$ for every $\theta$ with $i(\alpha_m) <\theta< i(\alpha_{m+1})$. 
Take any such an $m$ and a $\theta$. Since $p \in G$, we have
\begin{equation}\label{e39}
\{ g_{\eta} \ | \ \alpha_m<\eta<\alpha_{m+1} \} \subseteq A_{m+1}.
\end{equation}
There are two cases:

\emph{Case 1.} $\theta=i([\theta])$. Then $[\theta]<\alpha_{m+1}$, and $g_{[\theta]} \in A_{m+1}$. Since $\pi_{i(\alpha_{m+1})}$ is elementary,
\[
w(\theta)=\pi_{i([\theta])}(g_{[\theta]})=\pi_{i(\alpha_{m+1})}(g_{[\theta]}) \in \pi_{i(\alpha_{m+1})}(A_{m+1}).
\]

\emph{Case 2.} $s([\theta])\leq \theta<i([\theta])$. Then $[\theta]$ is a limit ordinal. By (\ref{e39}) and the characterization of genericity, i.e., Theorem \ref{itu}, $A_{m+1} \cap V[G]_{\kappa(g_{[\theta]})} \in \mathcal{F}(g_{[\theta]})$. Since $\pi_{\theta}$ is elementary, it follows that $\pi_{\theta}(A_{m+1} \cap V[G]_{\kappa(g_{[\theta]})}) \in \mathcal{F}(\pi_{\theta}(g_{[\theta]}))$. Note also that $\beta_{\theta}<\beta_{i([\theta])}$, by Lemma \ref{c5} (for the case $\theta=s(\delta)$, since the measure sequence $\pi_{s(\delta)}(u)$ is obtained from $\pi_{s(\delta),s(\delta)+1}$, this lemma is also true), we have
\[
w(\theta) \in \pi_{i([\theta])}(A_{m+1} \cap V[G]_{\kappa(g_{[\theta]})})=\pi_{i(\alpha_{m+1})}(A_{m+1} \cap V[G]_{\kappa(g_{[\theta]})}) \subseteq \pi_{i(\alpha_{m+1})}(A_{m+1}).
\]
Hence in any case, $w(\theta) \in \pi_{i(\alpha_{m+1})}(A_{m+1})$. So $\pi''G \subseteq H$, and therefore, we may lift $\pi$ and obtain an elementary embedding $\pi^+:V[G] \rightarrow N[H]$ by using Silver's criterion.

Let $D=\pi_{s(\delta)+1,i(\delta)}(\pi_{s(\delta),s(\delta)+1} ''(\pi_{s(\delta)}(\kappa')))$.
Since $\pi_{s(\delta),s(\delta)+1}$ witnesses that $\pi_{s(\delta)}(\kappa)$ is $\pi_{s(\delta)}(\kappa')$-supercompact, we have $\pi_{s(\delta),s(\delta)+1} ''(\pi_{s(\delta)}(\kappa')) \in N_{s(\delta)+1}$. Note also that $\pi_{s(\delta)+1,i(\delta)}$ is elementary, we have $D=\pi_{s(\delta)+1,i(\delta)}(\pi_{s(\delta),s(\delta)+1} ''(\pi_{s(\delta)}(\kappa'))) \in N$. Meanwhile,
$\pi''\kappa' \subseteq D$ and $N \models |D|<\pi(\kappa)$. So $\pi^+$ witnesses $\kappa$ is $\lambda$-strongly compact up to $\kappa'$.

Since $\mathcal{S}$ is a proper class and $\kappa' \in \mathcal{S}$ is arbitrary, $\kappa$ is $\lambda$-strongly compact in $V[G]$.

Now by Fact \ref{fact41}, we can see that $\kappa$ is the least $\lambda$-strongly compact cardinal in $V[G]$ (actually, $\kappa$ is the least $\omega_1$-strongly compact cardinal).

This concludes the proof of Theorem \ref{t1}.
\end{proof}

\begin{corollary}\label{c2s}
Suppose $\lambda<\kappa$ are supercompact cardinals. Then for every regular cardinal $\delta$ with $\lambda \leq \delta <\kappa$, there exists a generic extension of $V$, in which $\kappa$ is the least $\lambda$-strongly compact cardinal and has cofinality $\delta$.
\end{corollary}
\begin{proof}
Since $\lambda$ is a supercompact cardinal and $\delta$ is regular, it follows that $\delta$ is $\lambda$-measurable. By Theorem \ref{t1}, in some Radin generic extension that adds a Radin generic sequence of length $\delta$, $\kappa$ is the least $\lambda$-strongly compact cardinal and has cofinality $\delta$.
\end{proof}
The following proposition generalized \cite[Theorem 2.3]{BMO2014} and shows that our consistency
result (Theorem \ref{t1}) is optimal.
\begin{proposition}\label{pro3}
Suppose $\kappa$ is the least
$\lambda$-strongly compact cardinal and has cofinality $\delta$. Then $\delta$ is $\lambda$-measurable.
Namely, $\delta$ carries a $\lambda$-complete uniform ultrafilter over $\delta$.
\end{proposition}
\begin{proof}
We first prove that there exists a definable elementary embedding $j:V \rightarrow M$ with $M$ transitive such that $j(\kappa)>\sup(j''\kappa)$.
Since $\kappa$ is the least $\lambda$-strongly compact cardinal, it follows that $\kappa$ is a limit cardinal by Theorem \ref{lem3}, and for every $\gamma<\kappa$, there is an $\alpha_{\gamma}>\kappa$ such that
 $\gamma$ is not $\lambda$-strongly compact up to $\alpha_{\gamma}$. Let $\alpha=\sup(\{\alpha_{\gamma} \ | \ \gamma<\kappa \})^+$.
Then $\eta$ is not $\lambda$-strongly compact up to $\alpha$ for every $\eta<\kappa$.

By the $\lambda$-strong compactness of $\kappa$, there is a definable elementary embedding $j:V \rightarrow M$ with $M$ transitive, so that $\mathrm{crit}(j) \geq \lambda$ and there is a $D \in M$ such that
$j''\alpha \subseteq D$ and $M \models |D|<j(\kappa)$. If $\sup(j''\kappa)=j(\kappa)$, then there is a cardinal $\eta<\kappa$
such that $M \models |D|<j(\eta)$. Thus $j$ witnesses that $\eta$ is $\lambda$-strongly compact up to $\alpha$, a contradiction. Hence, $\sup(j''\kappa)<j(\kappa)$.

Now for such an embedding $j$, we claim that $\sup(j''\delta)<j(\delta)$.
If not, $j(\delta)=\sup(j''\delta)$. Since $\kappa$ is a limit cardinal and has cofinality $\delta$, there is an increasing cofinal sequence $\vec{\kappa}=\langle \kappa_{\alpha} \ | \ \alpha<\delta \rangle$ of cardinals converging to $\kappa$. By elementarity, $j(\vec{\kappa})$ is an increasing cofinal sequence on $j(\kappa)$ in $M$. Then in $V$, $\langle j(\vec{\kappa})(j(\alpha)) \ | \ \alpha<\delta \rangle$ is an increasing cofinal sequence on $j(\kappa)$ since $j(\delta)=\sup(j''\delta)$. By elementarity, $j(\vec{\kappa})(j(\alpha))=j(\vec{\kappa}(\alpha))=j(\kappa_{\alpha})$. So $\langle j(\kappa_{\alpha}) \ | \ \alpha<\delta \rangle$ is an increasing cofinal sequence on $j(\kappa)$, which means that $j(\kappa)=\sup(j''\kappa)$, a contradiction. Therefore, $j(\delta)>\sup(j''\delta)$.

Hence, $\delta$ carries a $\lambda$-complete uniform ultrafilter by the proof of Proposition \ref{pu}, and it is also $\lambda$-measurable.
\end{proof}
\subsection*{Acknowledgement}
We would like to thank Professor Joan Bagaria and Professor Liuzhen Wu for making useful suggestions, which improved clarity and readability of the paper.
The first author was supported by the China Scholarship Council (No. 202004910742) and NSFC No. 11871464.
The second author was supported by a UKRI Future Leaders Fellowship [MR/T021705/1].

\end{document}